\documentclass[12pt,a4paper,reqno]{amsart}

\usepackage[margin=2.5cm]{geometry}
\usepackage{graphics, amsmath,enumitem, comment, url}
\usepackage{graphicx}
\usepackage{epsfig}

\newif\ifcolorcomments
\newcommand{\allowcomments}[4]{
\newcommand{#1}[1]{\ifdraft{\ifcolorcomments{\textcolor{#4}{##1 --#3}}\else{\textsl{ ##1 \ --#3}}\fi}\else{}\fi}
}

\allowcomments{\commumtaz}{MH}{Mumtaz}{green}
\allowcomments{\comwang}{BW}{BWWang}{blue}
\allowcomments{\comkle}{DK}{DK}{magenta}
\allowcomments{\comnick}{NW}{Nick}{red}

\colorcommentstrue
\usepackage{amssymb}
\usepackage{amsmath,amsthm}
\usepackage{mathrsfs}
\usepackage{bbm}
\def\bc{\begin{center}}
\def\ec{\end{center}}
\def\be{\begin{equation}}
\def\ee{\end{equation}}

\def\N{\mathbb N}
\def\Z{\mathbb Z}
\def\Q{\mathbb Q}

\def\R{\mathbb R}

\def\H{\mathcal H}
\def\HH{\mathcal H}
\DeclareMathOperator{\diam}{diam}

\newcommand{\given}{\upharpoonleft}
\DeclareMathOperator{\Supp}{Supp}
\newcommand\hdim{\dim_{\mathcal H}}

\def\Z{\mathbb Z}

\newtheorem{lem}{Lemma}[section]

\newtheorem{pro}[lem]{Proposition}
\newtheorem{thm}[lem]{Theorem}

\newtheorem{cor}[lem]{Corollary}
\newtheorem{claim}[lem]{Claim}
\newtheorem{rem}[lem]{Remark}
\numberwithin{equation}{section}
\usepackage{colortbl}

\newif\ifdraft\drafttrue

\newcommand{\subscript}[2]{$#1 _ #2$}

\begin{document}



\subjclass[2010]  {}

\title[Hausdorff measure of Dirichlet non-improvable numbers]{The Generalised Hausdorff measure of sets of Dirichlet non-improvable numbers}

\author[P. Bos]{Philip Bos}

\address{Philip Bos, Department of Mathematical and Physical Sciences, La Trobe University, Bendigo 3552, Australia. }  \email{phil@philbos.com}

\author[M. Hussain]{Mumtaz Hussain}
\address{Mumtaz Hussain, Department of Mathematical and Physical Sciences, La Trobe University, Bendigo 3552, Australia. } \email{M.Hussain@latrobe.edu.au}
\author[D. Simmons]{David Simmons}
\address{David Simmons,  The University of York, England, UK}  \email{David.Simmons@york.ac.uk}


\maketitle

\begin{abstract} 

Let $\psi:\mathbb R_+\to\mathbb R_+$ be a non-increasing function. A real number $x$ is said to be $\psi$-Dirichlet improvable
if the system $$|qx-p|< \, \psi(t) \ \ {\text{and}} \ \ |q|<t$$
has a non-trivial integer solution for all large enough $t$. Denote the collection of such points by $D(\psi)$.  In this paper, we prove a zero-infinity law valid for all dimension functions under natural non-restrictive conditions. Some of the consequences are zero-infinity laws,   for all essentially sub-linear dimension functions proved by Hussain-Kleinbock-Wadleigh-Wang (2018),   for some non-essentially sub-linear dimension functions, and for all dimension functions but with a growth condition on the approximating function.




\end{abstract}

\section{ Introduction}

\noindent At its most fundamental level, the theory of Diophantine approximation is concerned with the question of how well a real number can be approximated by rationals. A qualitative answer is provided by the fact that the set of rational numbers is dense in the real numbers. Seeking a quantitative answer leads to the theory of metric Diophantine approximation. Dirichlet's theorem (1842) is the starting point in this theory.

%

\begin{thm}[Dirichlet, 1842]\label{Dir}
\label{Dirichletsv}
 \noindent Given $x\in \R$ and $t>1
$, there exist integers $p,q$
 such that
  \begin{equation*}\label{eqdir} \left\vert qx-p\right\vert\leq 1/t \quad{\rm and} \quad  1\leq{q}<{t}. \end{equation*}

\end{thm}

An easy consequence (known before Dirichlet) is the following global statement concerning the `rate' of rational approximation to any real number.
\begin{cor}
 \noindent For any $x\in \R$, there exist infinitely many integers $p$ and $q > 0 $ such that
\begin{equation*}\label{side1}
\left\vert qx-p\right\vert<1/q.
\end{equation*}
\end{cor}

A strengthening of this corollary is the classical $\Psi$-approximable set, here stated with a slight change of notation. Let $\Psi:[q_0, \infty)\to \mathbb R_+$  be a  non-decreasing function with $q_0\geq 1$  fixed. Then
\begin{equation*}
\mathcal K(\Psi):=\left\{x\in\mathbb [0, 1):  \left|x-\frac{p}{q}\right|\, <\, \frac1{q^2\Psi(q)}, \text{ for infinitely many} \ (p, q) \in \Z\times\N\right\}.
\end{equation*}
 A comprehensive metrical theory (Lebesgue measure, Hausdorff measure and Hausdorff dimension) for the set $\mathcal K(\Psi)$ is well-known, see for example \cite{BV06}. To transition into the statement of results we first introduce necessary notations.

\noindent{\bf Notation.} Here and throughout by a {\sl dimension function}, we mean an increasing, continuous function $f:\mathbb R_+\to \mathbb R_+$ such that $f(r)\to 0$ as $r\to 0$. For real quantities $A,B$ that
depend on parameters, we write  $A \lesssim B$ if $A \leq c B$ for a constant $c> 0$ that is independent of those parameters. 
 We write $B\asymp A$ if $A\lesssim B \lesssim A$. To simplify the presentation, we start by fixing some notation. We use $A\gg B$ to indicate that $|A/B|$ is sufficiently large.  We use $\mathcal L(\cdot)$,  $\hdim$ and $\mathcal{H}^f$ to denote the Lebesgue measure, Hausdorff dimension, and $f$-dimensional Hausdorff measure, respectively.

We state the most modern result from the seminal paper of Beresnevich-Velani \cite[Theorem 2]{BV06} with a slight improvement as noted in \cite[Theorem 2.6]{HKWW}.

\begin{thm}[Jarn\'ik, 1931]\label{Jarnik}
Let $\psi$ be a non-increasing function, and let $f$ be a dimension function satisfying 
the following properties:
\begin{equation}\label{f2infty}
\lim_{x\to 0}\frac{f(x)}{x}=\infty,
\end{equation} 
and \begin{equation}\label{ff2}
 \exists\, C \ge 1  \text{ such that } \frac{f(x_2)}{x_2}\le C\frac{f(x_1)}{x_1}  {\text{ whenever}}\ x_1< x_2\ll 1.
\end{equation} 
Then
\begin{equation*}\label{sum1} \H^f\big(\mathcal K(\Psi)\big)=\begin {cases}
 0 \ & {\rm if } \quad
 \sum\limits_{t} {t}f\left(\frac{1}{t^2\Psi(t)}\right) \, < \, \infty; \\[2ex]
 \infty \ & {\rm if } \quad
 \sum\limits_{t} {t}f\left(\frac{1}{{t^2\Psi({t})}}\right)\, = \, \infty.
\end {cases}\end{equation*}
\end{thm}

The condition \eqref{f2infty} is a natural condition used to exclude the case when $\H^f$ is comparable to the Lebesgue measure. In fact,  when $x^{-1}f(x)$ tends to a finite value as $x\to 0$ then $\H^f$ is comparable to the Lebesgue measure, see  \cite[Remark 3]{BV06}.  Note that the original Jarn\'ik theorem assumes the condition that ``$x^{-1}f(x)$ is decreasing'', the fact that it can be replaced with  the quasi-monotonicity condition given in \eqref{ff2} was noted in  \cite[Theorem 2.6]{HKWW}.  This condition is only required for the divergence case and the convergence case is free from any assumptions on the dimension  and the approximating functions.

%

%


Kleinbock-Wadleigh \cite{KlWad16} considered improvements to Dirichlet's theorem  (Theorem \ref{Dir}) by considering the following set
\begin{equation*}
D(\psi):=\left\{x\in\mathbb [0, 1):  \begin{aligned}&\exists\, N  \ {\rm such\  that\ the\ system}\ |qx-p|\, <\, \psi(t), \  
 |q|<t\  \\
  &\text{has a nontrivial integer solution for all }t>N\quad
                           \end{aligned}
\right\}.
\end{equation*}
Here and throughout, $\psi:\mathbb R_+\to\mathbb R_+$ denote a non-increasing function. For reference, a real number $x$ will
be called
\emph{$\psi$-Dirichlet improvable} if $x\in D(\psi)$, and if $x\in D^c(\psi)$ then it will be referred to as \emph{$\psi$-Dirichlet non-improvable}.   The main result of \cite{KlWad16} was a Lebesgue measure dichotomy statement. To state their result and other results of the paper, we introduce an auxiliary function

\begin{equation}\label{twopsis1}{
\Psi(t):=\frac{t\psi(t)}{1-t\psi(t)} = \frac{1}{1-t\psi(t)} - 1.
}\end{equation}

\begin{thm}[{\cite[Theorem 1.8]{KlWad16}}]
\label{KW} Let $\psi$ be non-increasing,  and suppose the function  $t\mapsto t\psi(t)$ be non-decreasing  and 
 $t\psi(t)<1$  for all  large $t$.  Then, 
\begin{equation*}\label{hdsum} \mathcal L\left(D^c(\psi)\right)=\begin {cases}
 0 \ & {\rm if } \quad\sum\limits_{q=1}^\infty\frac{\log{\Psi}(q)}{q{\Psi}(q)}
\, < \, \infty;  \\[2ex]
1 \ & {\rm if } \quad
\sum\limits_{q=1}^\infty\frac{\log{\Psi}(q)}{q{\Psi}(q)} \, = \, \infty.
\end {cases}\end{equation*}
\end{thm}

Our main result  is the following generalised $f$-Hausdorff measure dichotomy statement valid for a range dimension functions. 

\begin{thm}\label{thm1}  Let $\psi$ be a non-increasing positive function with $t\psi(t)<~1$ for all large $t$.  Let $f$ be a dimension function satisfying \eqref{f2infty} and \eqref{ff2}.
 Then
$$\H^f\big(D^c(\psi)\big)=\begin {cases}
 0 \ & {\rm if } \quad \sum\limits_{k=1}^\infty\  \sum\limits_{ j<\log_2\Psi(2^k)} 2^{2k+j}f\left(\frac{2^{-(2k+j)}}{\Psi(2^{k+j})}\right) \, < \, \infty; \\[2ex]
 \infty \ & {\rm if } \quad  \sum\limits_{k=1}^\infty \  \sum\limits_{ j<\log_2\Psi(2^k)} 2^{2k+j}f\left(\frac{2^{-(2k+j)}}{\Psi(2^{k+j})}\right)  \,= \, \infty.
\end {cases}$$

\end{thm}
The condition \eqref{f2infty}  comes into play for establishing the divergence case, in particular, our proof relies on a criterion from \cite{HussainSimmons2} which assumes this condition. As stated above, \eqref{f2infty} is a natural condition to exclude the Lebesgue measure case. 


\subsection{Some consequences of Theorem \ref{thm1}}
In \cite{HKWW}, the $\H^f$ measure for  the set $D^c(\psi)$ was proved for the  dimension functions satisfying the condition
\begin{equation}\label{ffm10}
\text{there exists }B>1\text{ such that } \limsup_{x\to 0}\frac{f(Bx)}{f(x)}<B.
\end{equation}
 They coined the term {\sl essentially sub-linear} (ESL) for such dimension functions. As noted in \cite[Lemma 3.1]{HKWW}, an ESL  dimension function satisfies \eqref{f2infty} and \eqref{ff2}.  This condition is clearly  satisfied for  $f(x)=x^s$ when $0\le s<1$.

\begin{cor}[{\cite[Theorem 1.6]{HKWW}}] \label{HKWWthm} Let $\psi$ be a non-increasing positive function with $t\psi(t)<~1$ for all large $t$. Let $f$ be a dimension function such that \eqref{ffm10} is satisfied. 
 Then
$$ \H^f\big(D^c(\psi)\big)=\begin {cases}
 0 \ & {\rm if } \quad \sum\limits_{q} {q}f\left(\frac{1}{{q^2\Psi({q})}} \right) \, < \, \infty; \\[2ex]
 \infty \ & {\rm if } \quad \sum\limits_{q} {q}f\left(\frac{1}{{q^2\Psi({q})}} \right)  \, = \, \infty.
\end {cases}$$
\end{cor}
\begin{rem}
In the proof of \cite[Theorem 1.6]{HKWW}, the ESL condition \eqref{ffm10} comes into play for establishing the equivalence of the series
\begin{equation}\label{equiv}
\sum_{q=1}^{\infty}qf\Big(\frac{1}{q^2\Psi(q)}\Big)\asymp\sum_{q=1}^{\infty}\sum_{1\le p\le q}f\Big(\frac{1}{pq\Psi(q)}\Big).
\end{equation}
The series on the right hand side arises as a covering argument for the set $D^c(\psi)$ and the series on the left hand side appears in Jarn\'ik's theorem \ref{Jarnik} stated above. \end{rem}

Note that, Corollary \eqref{HKWWthm} is not valid for dimension functions that do not satisfy the condition \eqref{ffm10} such as $f(x) = x$ or $f(x)= x \log (1/x)$.  However, even if Corollary \ref{HKWWthm} holds for $f(x)=x$,  the sum condition $\sum\frac1{q\Psi(q)}$, is weaker than the sum condition of Theorem \ref{KW},  $\sum_{q}\frac{\log{\Psi}(q)}{q{\Psi}(q)}$, by a `$\log$' factor.  

The  following corollary shows that Theorem \ref{thm1} is also valid for some {\sl non-essentially sub-linear } (NESL) dimension functions. A dimension function $f$ is NESL if it satisfies
\begin{equation}\label{NESL}
\text{for every }B>1, \limsup_{x\to 0}\frac{f(Bx)}{f(x)}\geq B.
\end{equation}
One example for NESL dimension function is
\begin{equation}\label{fnesl}f(x)=x(\log(1/x))^{a_1}(\log\log(1/x))^{a_2}\cdots (\log\log\cdots\log(1/x))^{a_n}.
\end{equation} where   $a_i\in\R_{\geq0}$ for $1\leq i\leq n$, and  there is at least one $1\leq i\leq n$ such that $a_i\neq 0$.  

\begin{cor}\label{cor1} Let $\psi$ be a non-increasing positive function with $t\psi(t)<~1$ for all large $t$. Let $f$ be given in \eqref{fnesl}, then
$$\H^f\big(D^c(\psi)\big)=\begin {cases}
 0 \ & {\rm if } \quad \sum\limits_{q} q\log\left(\Psi(q)\right)f\left(\frac{1}{{q^2\Psi({q})}} \right) \, < \, \infty; \\[2ex]
 \infty \ & {\rm if } \quad \sum\limits_{q}q \log\left(\Psi(q)\right)f\left(\frac{1}{{q^2\Psi({q})}} \right)  \, = \, \infty.
\end {cases}$$
\end{cor}
%

The next corollary shows that the conditions \eqref{ffm10} or \eqref{NESL} in Corollaries \ref{HKWWthm} and \ref{cor1} respectively can be dropped at a cost of a restriction on the approximating function $\Psi$. 
\begin{cor}\label{cor2} Let $\psi$ be a non-increasing positive function with $t\psi(t)<~1$ for all large $t$. Let $f$ be a dimension function satisfying \eqref{f2infty} and \eqref{ff2}. Let $\Psi$ be as in \eqref{twopsis1} such that, for all $ x>0$ and $Q>1$, the following condition holds
\begin{equation}\label{condpsi}
\Psi(Q^x)\asymp\Psi(Q),
\end{equation}
where the implied constant depends only on $x$.
 Then
$$\H^f\big(D^c(\psi)\big)=\begin {cases}
 0 \ & {\rm if } \quad \sum\limits_{q} q\log\left(\Psi(q)\right)f\left(\frac{1}{{q^2\Psi({q})}} \right) \, < \, \infty; \\[2ex]
 \infty \ & {\rm if } \quad \sum\limits_{q}q \log\left(\Psi(q)\right)f\left(\frac{1}{{q^2\Psi({q})}} \right)  \, = \, \infty.
\end {cases}$$
\end{cor}

Note that the condition \eqref{condpsi} is trivially satisfied for approximating function such as $\Psi(Q)=(\log(Q))^s$ for any $s\geq 0$. 


%
%
From the corollaries it should be clear that Theorem \ref{thm1} is valid for a range of dimension functions. However, there are certain dimension functions such as $f(x)=xe^{(\log\frac1x)^\beta}$ for $\beta<1$, for which we are unable to extract conclusive information (from Theorem \ref{thm1}) for the Hausdorff measure of the set $D^c(\psi)$.

We refer the reader to \cite{KimKim, KW2} for the metrical theory of the higher dimensional affine form version of the set $D^c(\psi)$.

\noindent{\bf Acknowledgements.} The second and third named authors are supported by the Australian Research Council Discovery Project (ARC DP200100994). The third named author is a Royal Society University Research Fellow. We thank Professor Dmitry Kleinbock  and Professor Baowei Wang for many useful discussions. Finally, we thank an anonymous referee for a careful reading of the paper, pointing out some pertinent questions that lead us to correct some of the proofs.  

\section{Continued fractions and Dirichlet improvability}
The starting point for the work of Davenport \& Schmidt \cite{DaSc70} and Kleinbock \& Wadleigh \cite{KlWad16} is an observation that Dirichlet improvability is equivalent to a condition on the growth rate of partial quotients. To recall this connection, we start off with some of the basic properties of continued fractions.

  Define the Gauss transformation $T:[0,1)\to [0,1)$  by
\[T(0):=0, \quad T(x):=\frac1x\ {\rm(mod}\ 1),\quad {\rm for} \ x\in (0, 1).\]
 Every $x\in [0,1)$ has a  {\sl continued fraction expansion}, $$x
=\frac{1}{a_1(x)+\displaystyle{\frac{1}{a_2(x)+\displaystyle{\frac{1}{a_3(x)+\ddots}}}}}:=[a_1(x),a_2(x),\dots]$$ where $a_1,a_2,\dots$ are positive integers called the {\sl partial quotients} of $x$   and  $a_n(x)=~\lfloor 1/T^{n-1}(x)\rfloor$ for each $n\ge 1.$ We also write $p_n/q_n = [a_1,...,a_n]$ ($p_n, q_n$ coprime) for the $n$'th  {\sl convergent} of $x$.  With the conventions
$p_{-1}=1, ~q_{-1}=0, ~p_0=0$ and  $~q_0=1$,  these sequences can be generated by the following recursive relations, see \cite{Khi_64} for further details,
\begin {equation}\label{recu}
p_{n+1}=a_{n+1}(x)p_n+p_{n-1}, \ \
q_{n+1}=a_{n+1}(x)q_n+q_{n-1},\ \  n\geq 0.
\end {equation}

Thus $p_n=p_n(x), q_n=q_n(x)$ are determined by the partial quotients $a_1,\dots,a_n$, so we may write $p_n=p_n(a_1,\dots, a_n), q_n=q_n(a_1,\dots,a_n)$. When it is clear which partial quotients are involved, we denote them by $p_n, q_n$ for simplicity.

%

For any integer vector $(a_1,\dots,a_n)\in \N^n$ with $n\geq 1$, write
\begin{equation*}\label{cyl}
I_n(a_1,\dots,a_n):=\left\{x\in [0, 1): a_1(x)=a_1, \dots, a_n(x)=a_n\right\}
\end{equation*}
for the corresponding `cylinder of order $n$', i.e.\  the set of all real numbers in $[0,1)$ whose continued fraction expansions begin with $(a_1, \dots, a_n).$

We will frequently use the following well known properties of continued fraction expansions.  They are explained in the standard texts \cite{IosKra_02,  Khi_64}.

\begin{pro}\label{pp3} For any {positive} integers $a_1,\dots,a_n$, let $p_n=p_n(a_1,\dots,a_n)$ and $q_n=~q_n(a_1,\dots,a_n)$ be defined recursively by \eqref{recu}. {Then:}
\begin{enumerate}[label={\rm (\subscript{\rm P}{\arabic*})}]
\item
\begin{eqnarray*}
I_n(a_1,a_2,\dots,a_n)= \left\{
\begin{array}{ll}
         \left[\frac{p_n}{q_n}, \frac{p_n+p_{n-1}}{q_n+q_{n-1}}\right)     & {\rm if }\ \
         n\ {\rm{is\ even}};\\
         \left(\frac{p_n+p_{n-1}}{q_n+q_{n-1}}, \frac{p_n}{q_n}\right]     & {\rm if }\ \
         n\ {\rm{is\ odd}}.
\end{array}
        \right.
\end{eqnarray*}
{\rm Thus, its length is given by} \begin{equation*}\label{lencyl}
\frac{1}{2q_n^2}\leq |I_n(a_1,\ldots,a_n)|=\frac{1}{q_n(q_n+q_{n-1})}\leq \frac1{q_n^2},
\end{equation*}
{\rm since} $$
 p_{n-1}q_n-p_nq_{n-1}=(-1)^n, \ {\rm for \ all }\ n\ge 1.
 $$
%
%
%

\item \begin{equation*}\label{p7}
\frac{1}{3a_{n+1}(x)q^2_n(x)}\, <\, \Big|x-\frac{p_n(x)}{q_n(x)}\Big|=\frac{1}{q_n(x)(q_{n+1}(x)+T^{n+1}(x)  q_n(x))}\, < \,\frac{1}{a_{n+1}q^2_n(x)}.
\end{equation*}
\end{enumerate}
\end{pro}

The results of  \cite{DaSc70, KlWad16} rely crucially on the following observation, \begin{align*}
x\in D(\psi) &\Longleftrightarrow |q_{n-1}x-p_{n-1}| \,<\, \psi(q_n)  \ {\text{for all}}\ n\gg 1 \\
&\Longleftrightarrow [a_{n+1}, a_{n+2},\dots]\cdot [a_n, a_{n-1},\dots, a_1]\, < \, \frac1{\Psi(q_n)}
  \ {\text{for all}}\ n\gg 1.\end{align*}
Where the {auxiliary} function $\Psi$ is defined in \eqref{twopsis1}. 
This leads to the following criterian for Dirichlet improvability.

\begin{lem}[\cite{KlWad16}, Lemma 2.2)]\label{kwlem}Let $x\in [0, 1)\smallsetminus\Q$, and let $\psi:[t_0, \infty)\to\R_+$ be non-increasing. Then
\begin{itemize}
\item [{\rm (i)}] $x$ is $\psi$-Dirichlet improvable 
 if $a_{n+1}(x)a_n(x)\, \le\,\Psi(q_n)/4$ for all sufficiently large $n$.
\item [{\rm (ii)}]$x$ is $\psi$-Dirichlet non-improvable
 if $a_{n+1}(x)a_n(x)\, >\, \Psi(q_n)$ for infinitely many~$n$.
\end{itemize}

\end{lem}

 As a consequence of this lemma, and some elementary properties of continued fractions (see \cite[\S2]{HKWW}, we have  the inclusions
\begin{equation*}\label{l1.2}
\mathcal K(3\Psi)\subseteq G_1(\Psi)\subset G(\Psi) \subset D^c(\psi)\subset G(\Psi/4),\end{equation*}
where
\begin{align*}\label{gpsi}
G(\Psi)&:=\Big\{x\in [0,1): a_n(x)a_{n+1}(x)\,>\, \Psi(q_n(x))  \ {\text{ for i.m.}}\ n\in \N\Big\},\\
G_1(\Psi)&:= \left\{x\in[0, 1): a_{n+1}(x)\,>\, \Psi(q_n(x))  \ \ {\rm for \ i.m.\ }n\in\N\right\},\\
\mathcal K(3\Psi)&:=\left\{x\in[0,1): \left|x-\frac pq\right|<\frac{1}{3q^2\Psi(q)} \ {\rm for \ i.m. \ } (p, q)\in \Z\times \N \right\}.
\end{align*}

Hence Theorems \ref{HKWWthm}  and \ref{thm1} can be restated in terms of the set $G(\Psi)$ as:

\begin{thm}[{\cite{HKWW}, 2018}] \label{HKWW2} Let $\Psi: [t_0, \infty)\to \R_+$ be a non-decreasing function.  Let $f$ be a dimension function such that \eqref{ffm10} is satisfied. Then
$$ \H^f\big(G(\Psi)\big)=\begin {cases}
 0 \ & {\rm if } \quad \sum\limits_{q} {q}f\left(\frac{1}{{q^2\Psi({q})}} \right) \, < \, \infty; \\[2ex]
 \infty \ & {\rm if } \quad \sum\limits_{q} {q}f\left(\frac{1}{{q^2\Psi({q})}} \right)  \, = \, \infty.
\end {cases}$$

\end{thm}

\begin{thm}\label{thm3} Let $\Psi: [t_0, \infty)\to \R_+$ be a non-decreasing function. Let $\psi$ be a non-increasing positive function with $t\psi(t)<~1$ for all large $t$.  Let $f$ be a dimension function satisfying \eqref{f2infty} and \eqref{ff2}. Then
$$\H^f\big(G(\Psi)\big)=\begin {cases}
 0 \ & {\rm if } \quad \sum\limits_{k=1}^\infty\  \sum\limits_{ j<\log_2\Psi(2^k)} 2^{2k+j}f\left(\frac{2^{-(2k+j)}}{\Psi(2^{k+j})}\right) \, < \, \infty; \\[2ex]
 \infty \ & {\rm if } \quad  \sum\limits_{k=1}^\infty \  \sum\limits_{ j<\log_2\Psi(2^k)} 2^{2k+j}f\left(\frac{2^{-(2k+j)}}{\Psi(2^{k+j})}\right)  \,= \, \infty.
\end {cases}$$

%
%
%
%

\end{thm}

It can readily be seen that the divergence case of Jarnik's theorem played a pivotal role in proving the divergence case of the Theorem \ref{HKWWthm}.  Further, it is worth pointing out that the difference set 
$G(\Psi)\setminus \mathcal K(3\Psi)$ is non-trivial as the Hausdorff dimension of this set is the same as that of $G(\psi)$ as proved in \cite{BBH}. We refer the reader to \cite{BHKW, HWX} for the Lebesgue measure and Hausdorff dimension results for a generalised form of the  set $G(\Psi)$.

For any $\tau\geq 0$, as a consequence of Theorem \ref{HKWW2},  the Hausdorff dimension of the sets of Dirichlet non-improvable numbers with prescribed order of approximation $\tau$ can be deduced. Define the set with \emph{order $\tau$} as
$$G(\tau):=\left\{x\in [0,1): \limsup_{n\to\infty}\frac{\log\left(a_n(x)a_{n+1}(x)\right)}{\log q_n(x)}\geq \tau\right\}$$
and \emph{exact order $\tau$} as 
$$G(\tau_{\rm exac}):=\left\{x\in [0,1): \limsup_{n\to\infty}\frac{\log\left(a_n(x)a_{n+1}(x)\right)}{\log q_n(x)}= \tau\right\}.$$
\begin{cor}
\begin{align*}\dim_\mathcal H G(\tau)=G(\tau_{\rm exac})=\frac{2}{\tau+2}.
\end{align*}

\end{cor}

We refer the reader to \cite{HuangWu} for the Hausdorff dimension of level sets obtained by replacing $\limsup$ with $\lim$ in the definition of  $G(\tau_{\rm exac})$.

%
%

Finally, for completeness we give a very brief introduction to Hausdorff measures and dimension.  For further details we refer to the beautiful texts \cite{BeDo_99, F_14}.

\subsection{Hausdorff measure and dimension}\label{HM}\

Let $f$ be a dimension function and let
$E\subset \R^n$.
 Then, for any $\rho>0$ a countable collection $\{B_i\}$ of balls in
$\R^n$ with diameters $\mathrm{diam} (B_i)\le \rho$ such that
$E\subset \bigcup_i B_i$ is called a $\rho$-cover of $E$.
Let
\[
\H_\rho^f(E)=\inf \sum_i f\big(\mathrm{diam}(B_i)\big),
\]
where the infimum is taken over all possible $\rho$-covers $\{B_i\}$ of $E$. It is easy to see that $\H_\rho^f(E)$ increases as $\rho$ decreases and so approaches a limit as $\rho \rightarrow 0$. This limit could be zero or infinity, or take a finite positive value. Accordingly, the \textit{Hausdorff $s$-measure $\H^f$} of $E$ is defined to be
\[
\H^f(E)=\lim_{\rho\to 0}\H_\rho^f(E).
\]

It is easily verified that Hausdorff measure is monotonic and countably sub-additive, and that $\H^f(\varnothing)=0$. Thus it is an outer measure on $\R^n$.

In the case when $f(x)=x^s$ for some $s\geq 0$, we write $\mathcal H^s$ for $\mathcal H^f$. Furthermore, for any subset $E$ one can verify that there exists a unique critical value of $s$ at which $\H^s(E)$ `jumps' from infinity to zero. The value taken by $s$ at this discontinuity is referred to as the \textit{Hausdorff dimension of  $E$} and  is denoted by $\hdim E $; i.e.,
\[
\hdim E :=\inf\{s\ge 0:\; \H^s(E)=0\}.
\] When $s=n$,  $\H^n$ coincides with standard Lebesgue measure on $\R^n$.

Computing Hausdorff dimension of a set is typically accomplished in two steps: obtaining the upper and lower bounds separately.

Upper bounds often can be handled by finding appropriate coverings. When dealing with a limsup set, one 
 {usually applies} the Hausdorff measure version of the famous Borel-Cantelli lemma (see Lemma 3.10 of \cite{BeDo_99}).

\begin{pro}\label{bclem}
    Let $\{B_i\}_{i\ge 1}$ be a sequence of measurable  sets in $\R^n$ and suppose that for some dimension function $f$,  $\sum_i f\big(\mathrm{diam}(B_i)\big) \, < \, \infty.$ Then  $\H^f(
    {\limsup_{i\to\infty}B_i})=0.$
\end{pro}

\section{Proof of  Theorem \ref{thm3}} 

\subsection{The convergence case} We are given that the series
\begin{equation}\label{star}
\sum\limits_{k=1}^\infty\  \sum\limits_{ j<\log_2\Psi(2^k)} 2^{2k+j}f\left(\frac{2^{-(2k+j)}}{\Psi(2^{k+j})}\right)
\end{equation}
converges. We can assume that $ \Psi(t)\ge 1$ for all $t\gg 1$. Otherwise, $\Psi(t)<1$ for all  large $t$ since we have assumed $\Psi$ to be non-decreasing. Then it is obvious that the set
$$G_1(\Psi)= \left\{x\in[0, 1): a_{n+1}(x)\,>\, \Psi(q_n)  \ \ {\rm for \ i.m.\ }n\in\N\right\}$$ and thus $$G(\Psi)=\Big\{x\in [0,1):  a_{n}(x)a_{n+1}(x)\ge \Psi(q_n) \ \ {\rm for \ i.m.}\ n\in \N\Big\},$$ contains all irrational numbers in $[0,1]$, and that the sum \eqref{star} diverges.   
Since $\Psi$ is increasing and from \eqref{recu} it follows that $q_n\geq a_nq_{n-1}$. We notice some obvious inclusions, 
\begin{align*}
G(\Psi)=& \Big\{x\in [0,1):  a_{n}(x)a_{n+1}(x)\ge \Psi(q_n) \ \ {\rm for \ i.m.}\ n\in \N\Big\}\\ \subseteq& \Big\{x\in [0,1):  a_{n}(x)a_{n+1}(x)\ge \Psi(a_nq_{n-1}) \ \ {\rm for \ i.m.}\ n\in \N\Big\}\\
   \subseteq&\bigcup_{n=N}^{\infty}\bigcup_{a_1,\dots,a_n} \bigcup_{a_{n+1} > \frac{\Psi(a_nq_{n-1})}{a_n}}I_{n+1}(a_1,\dots,a_n,a_{n+1})\\ &
    = \mathcal A_1(\Psi)\cup \mathcal A_2 (\Psi).
\end{align*}
Where \begin{align*}
\mathcal A_1 (\Psi)&=\bigcup_{n=N}^{\infty}\bigcup_{a_1,\dots,a_n}\bigcup_{a_n \leq \Psi(q_{n-1})} \bigcup_{a_{n+1} > \frac{\Psi(a_nq_{n-1})}{a_n}}I_{n+1}(a_1,\dots,a_n,a_{n+1}),\\ 
\mathcal A_2(\Psi)&=\bigcup_{n=N}^{\infty}\bigcup_{a_1,\dots,a_n}\bigcup_{a_n >\Psi(q_{n-1})}  \bigcup_{a_{n+1} > \frac{\Psi(a_nq_{n-1})}{a_n}}I_{n+1}(a_1,\dots,a_n,a_{n+1}).
\end{align*}
\subsubsection{Covering for $\mathcal A_1(\Psi)$}

To estimate the Hausdorff measure of the set $\mathcal A_1(\Psi)$, we proceed as follows.  Let
$$J_n(a_1,\dots,a_n):=\bigcup_{a_{n+1} > \frac{\Psi(a_nq_{n-1})}{a_n}}I_{n+1}(a_1,\dots,a_n,a_{n+1}).$$
Using (P$_1$) in Proposition \ref{pp3} and the recursive relation \eqref{recu},
the diameter of $J_n(a_1,\dots,a_n)$ can be bounded {as follows}:
 \begin{align*}
|J_n(a_1,\dots,a_n)|&=\sum_{a_{n+1} > \frac{\Psi(a_nq_{n-1})}{a_n}}\left|\frac{a_{n+1}p_n+p_{n-1}}{a_{n+1}q_n+q_{n-1}}-\frac{(a_{n+1}+1)p_n+p_{n-1}}{(a_{n+1}+1)q_n+q_{n-1}}\right|\\
&\le  \left|\frac{\frac{\Psi(a_nq_{n-1})}{a_n}p_n+p_{n-1}}{\frac{\Psi(q_n)}{a_n}q_n+q_{n-1}}-\frac{p_n}{q_n}\right|=
\frac{1}{\left(\frac{\Psi(a_nq_{n-1})}{a_n}q_n+q_{n-1}\right)q_n}\\
&\le \frac{1}{\Psi(a_nq_{n-1})a_nq_{n-1}^2}.\end{align*}

Let $Q>1$ and $Q<q_{n-1}\leq 2Q$. Then we can bound the diameter of  $J_n(a_1,\dots,a_n)$ as
\begin{equation*}
|J_n(a_1,\dots,a_n)|\ll  \frac{1}{\Psi(a_nQ)a_nQ^2}.
\end{equation*} 

Hence, the cost of the cover when $a_{n}<\Psi(q_{n-1})$,  is \[\sum_{a=1}^{\Psi(Q)} f \left(\frac{1}{aQ^2\Psi(aQ)}\right). \] In the case  $a_{n}>\Psi(q_{n-1})$,  the cost of the cover is given by \[ f \left(\frac{1}{Q^2 \Psi(Q)}\right).\]

%
%
%
%
%
%
%
%
%
%
%
%
%
%
%
%
%
%
%
%

Since  $Q>1$, it follows from equation \eqref{lencyl} that for each window $[Q,2Q]$, there are at most $Q^2$ cylinders $I_n$ of length comparable (up to a constant) to $~Q^{-2}$.
 Multiplying the cost of the cover given above by $Q^2$ which are the number of such intervals, and then summing over all the windows $Q=~2^k,$ we have

\[\sum_{Q=2^k; k\geq 1} Q^2\sum_{a=1}^{\Psi(Q)} f \left(\frac{1}{aQ^2\Psi(aQ)}\right) + \sum_{Q=2^k; k\geq 1} Q^2f \left(\frac{1}{Q^2 \Psi(Q)}\right).\]

Applying Cauchy's condensation test on the second term, and rewriting the first term gives the total cost as

\[\sum_{\substack{k\geq 1\\ Q=2^k }}  \  \sum_{ \substack{j\geq 1, A=2^j \\ A<\Psi(Q)} } Q^2 A f \left( \frac1{Q^2 A \Psi(Q A)}\right)+ \sum_q q f \left(\frac{1}{q^2 \Psi(q)}\right).\]

%
%
The second term is clearly smaller than the first term  so we can ignore it. 
Thus, using Proposition \ref{bclem}, we have

$$\H^f (\mathcal A_1(\Psi)) =0 \quad \text{if}\quad \sum\limits_{k=1}^\infty\  \sum\limits_{ j<\log_2\Psi(2^k)} 2^{2k+j}f\left(\frac{2^{-(2k+j)}}{\Psi(2^{k+j})}\right)<\infty.$$

\subsubsection{Covering for $\mathcal A_2(\Psi)$}
Next we notice  that the set $\mathcal A_2(\Psi)$ is a subset of the Jarn\'ik set, that is, it is contained in the set $G_1(\Psi)$. To see this note that if $a_n \geq\Psi(q_{n-1})$ then it follows that 
$a_{n+1} \geq\Psi(q_{n-1})$ for infinitely many $n$.  This in turn implies that for any dimension function $f$
$$\H^f (\mathcal A_2(\Psi)) \leq \H^f ( G_1(\Psi))\leq \H^f (\mathcal K(\Psi))=0 \iff  \sum_{q}qf\left(\frac1{q^2\Psi(q)}\right)<\infty.$$

However notice that the series above is smaller than the one we are claiming in our theorem.  Hence combining both the coverings for $\mathcal A_1(\Psi)$ and $\mathcal A_2(\Psi)$, and by using Proposition \ref{bclem}, we conclude that   
 \[\H^f(G(\Psi))=0 \quad \text{if}\quad  \sum\limits_{k=1}^\infty\  \sum\limits_{ j<\log_2\Psi(2^k)} 2^{2k+j}f\left(\frac{2^{-(2k+j)}}{\Psi(2^{k+j})}\right)<\infty.\]

\subsection{The divergence case}
For the divergence case, we appeal to a recent criterion introduced by Hussain-Simmons \cite{HussainSimmons2}. We briefly state this criterion below and then use it to prove the divergence case in the subsequent subsection.

\subsubsection{A generalised Hausdorff measure criterion} Let $X$ be a metric space.   For $\delta>0$, a measure $\mu$ is {\em Ahlfors $\delta$-regular} if and only if there exist positive constants $0<c_1<1< c_2<\infty$ and $r_0 > 0$ such that the inequality 
$$c_1r^\delta\leq \mu (B(x,r)) \leq c_2 r^\delta$$ 
holds for every ball $B:=B(x, r)$ in $X$ of radius $r \leq r_0$ centred at $x \in \Supp(\mu)$, where $\Supp(\mu)$ denotes the topological support of $\mu$. The space $X$ is called Ahlfors $\delta$-regular if there is an Ahlfors $\delta$-regular measure whose support is equal to $X$. If $X$ is Ahlfors $\delta$-regular, then so is the $\HH^\delta$ measure restricted to $X$ i.e.  $\HH^\delta\given X$.   


\begin{thm}[Hussain-Simmons, \cite{HussainSimmons2}]\label{mainthm}
Fix $\delta > 0$, let $(B_i)_i$ be a sequence of open sets in an Ahlfors $\delta$-regular metric space $X$, and let $f$ be a dimension function such that
 $r^{-\delta} f(r) \to \infty$  as $r \to 0$ and
 \begin{equation}\label{ff3}
 \exists\, \kappa \ge 1  \text{ such that } \frac{f(x_2)}{x_2^\delta}\le \kappa\frac{f(x_1)}{x_1^\delta}  {\text{ whenever}}\ x_1< x_2\ll 1.
\end{equation}
Fix $C > 0$, and suppose that the following hypothesis holds:
\begin{itemize}
\item[(*)] For every ball $B_0 \subset X$ and for every $N\in\N$, there exists a probability measure $\mu = \mu(B_0,N)$ with $\Supp(\mu) \subset \bigcup_{i\geq N} B_i\cap B_0$, such that for every ball $B = B(x,\rho) \subset X$, we have
\begin{equation}
\label{muBbound}
\mu(B) \lesssim \max\left(\left(\frac{\rho}{\diam B_0}\right)^\delta,\frac{f(\rho)}{C}\right).
\end{equation}
\end{itemize}
Then for every ball $B_0$,
\[
\HH^f\left(B_0 \cap \limsup_{i\to\infty} B_i\right) \gtrsim C.
\]
In particular, if the hypothesis \text{(*)} holds for all $C$, then
\[
\HH^f\left(B_0\cap \limsup_{i\to\infty} B_i\right) = \infty.
\]
\end{thm}
Note that, originally Theorem \ref{mainthm} was proved with the condition \eqref{ff3} for $\kappa=1$ only, that is,  ``$r\mapsto r^{-\delta}f(r)$ is decreasing'' was assumed. However, this monotonicity condition can be replaced with the ``quasi-monotonicity'' condition as in \eqref{ff3}. Within the proof of their theorem, the condition was only used on the last step,  which is still valid with the assumption \eqref{ff3}.

The condition ``$r^{-\delta} f(r) \to \infty$  as $r \to 0$'' is a natural condition which implies that $\HH^f(B)=\infty$. The hypothesis (*) is the  main ingredient of this theorem and, roughly speaking, this gives a systematic way of constructing the probability measure on the limsup set. 


\subsubsection{Proof of the divergence case} Now we are in a position to prove the divergence case. We will prove, in particular, the following generalised form from which the divergence case of Theorem \ref{thm1} readily follows.

For each $a\in E$, let $u_a:[0, 1]\to [0, 1]$ defined by $u_a(x)=\frac1{a+x}$. The collection of maps $u=(u_a)_{a\in E}$ is a Gauss Iterated Function System (GIFS)  if:
\begin{itemize}
\item $E\subseteq \N$ is a (finite or infinite) index set, which is referred to as an alphabet;
\item $ X\subseteq\R$ is a nonempty compact set which is equal to the closure of its interior;
\item for all $a\in E, u_a(X)\subset X$.
\end{itemize}

\begin{thm}\label{thm2}
Let $(u_a)_{a\in E}$ be the Gauss iterated function system. For each finite word $\omega \in E^*$ and $a \leq \Psi(Q_\omega)$ let
\[
S_{\omega a} = u_{\omega a}([0,a/\Psi(Q_\omega a)]).
\]
Let $f$ be a dimension function such that $\sum_{\omega,a} f(\diam S_{\omega a})$ diverges. Then
\[
\HH^f\left(\limsup_{\omega,a} S_{\omega a}\right) = \infty.
\]
\end{thm}
 To apply theorem \ref{thm2} to derive the Hausdorff measure of the set $G(\Psi)$, notice that the set $S_{\omega a}$ corresponds to the set $$\bigcup_{a_n \leq \Psi(q_{n-1})} \bigcup_{a_{n+1} > \frac{\Psi(a_nq_{n-1})}{a_n}}I_{n+1}(a_1,\dots,a_n,a_{n+1}).$$ Hence $$\limsup_{\omega,a} S_{\omega a}\subseteq G(\Psi).$$ 
   As in the convergence case for the set $\mathcal A_1(\Psi)$, let $Q>1$ and $Q<q_{n-1}\leq 2Q$. Then  we have  
 \begin{align*}\infty =\sum_{\omega,a} f(\diam S_{\omega a}) &\asymp \sum_{Q=2^k; k\geq 1}Q^2 \sum_{a=1}^{\Psi(Q)} f \left(\frac{1}{aQ^2\Psi(aQ)}\right)\\ &\asymp \sum\limits_{k=1}^\infty\  \sum\limits_{ j<\log_2\Psi(2^k)} 2^{2k+j}f\left(\frac{2^{-(2k+j)}}{\Psi(2^{k+j})}\right).\end{align*}

\begin{proof}[Proof of Theorem \ref{thm2}]
Fix $B_0 \subset [0,1]$ and $N\in\N$, and we will construct the measure $\mu = \mu(B_0,N)$ such that the hypothesis (*) in Theorem \ref{mainthm} holds. 

For each $k,\ell\in \N$ let
\[
A_{k,\ell} = \{(\omega,a) : 2^k \leq Q_\omega < 2^{k+1}, \;\; 2^\ell \leq a < 2^{\ell + 1}\}.
\]
Then for all $(\omega,a)\in A_{k,\ell}$, we have $\diam(S_{\omega a}) \asymp \rho_{k,\ell}$. Thus
\[
\sum_{k,\ell} \#(A_{k,\ell}) f(\rho_{k,\ell}) = \infty.
\]

Let $$A_{k,\ell}' = \{(\omega,a) \in A_{k,\ell} : S_{\omega a} \subset B_0\}.$$

\begin{claim}
\[
\#(A_{k,\ell}') \gtrsim \#(A_{k,\ell}) |B_0|
\]
for all $k,\ell \geq N_0$, for some sufficiently large $N_0$.
\end{claim}
\begin{proof}
Indeed, consider the set
\[
\left\{\tau b : Q_\tau < 2^k \leq Q_{\tau b}\right\}.
\]
Fix $\tau$ such that $Q_\tau < 2^k$. Since $Q_{\tau b} = b Q_\tau + Q_{\tau'}$, where $\tau'$ is $\tau$ minus its last symbol, it follows that there are approximately $2^k / Q_\tau$ values of $b$ such that $2^k \leq Q_{\tau b} < 2^{k + 1}$. On the other hand, the set
\begin{equation}
\label{disjoint}
\bigcup_{b: 2^k \leq Q_{\tau b}} u_{\tau b}([0,1])
\end{equation}
has measure approximately $(2^k / Q_\tau)^{-1} Q_\tau^{-2} = 2^{-k} / Q_\tau$. Since the sets \eqref{disjoint} form a disjoint cover of $[0,1]$, it follows that
\begin{align*}
\#(A_{k,\ell}') &\geq 2^\ell \#\{\omega: 2^k \leq Q_\omega < 2^{k+1}, \; S_\omega\cap B_0 \neq \emptyset\}
\\ &\geq 2^\ell \sum_{\substack{\tau \\ Q_\tau < 2^k \\[1ex] S_\tau\cap B_0 \neq \emptyset}} \left|\bigcup_{b: 2^k \leq Q_{\tau b}} u_{\tau b}([0,1])\right|\\[1ex]
&\asymp 2^\ell |B_0| \#\left\{\tau : Q_\tau < 2^k, \; S_\tau\cap B_0 \neq \emptyset\right\}\\
&\asymp \#(A_{k,\ell}) |B_0|.
\end{align*}
{\renewcommand{\qedsymbol}{$\triangleleft$}}\end{proof}

It follows that
\[
\sum_{k,\ell \geq N_0} \#(A_{k,\ell}') f(\rho_{k,\ell}) = \infty.
\]
Fix $N_1$ such that
\[
\Omega = \sum_{N_0 \leq k,\ell \leq N_1} \#(A_{k,\ell}') f(\rho_{k,\ell}) \geq C
\]
and define the measure $\mu$ as follows:
\[
\mu = \frac1\Omega \sum_{N_0 \leq k,\ell\leq N_1} \sum_{(\omega,a) \in A_{k,\ell}'} f(\rho_{k,\ell}) \lambda_{S_{\omega a}},
\]
where $\lambda_A$ is normalized Lebesgue measure on a set $A$.

Let $B = u_\tau([1/b_1,1/b_2])$ for some $\tau,b_1,b_2$ (possibly $b_2 = 1$). Next we estimate $\mu(B)$ and show that it satisfies \eqref{muBbound}. Let $$A_{k,\ell}'' = \{(\omega,a) \in A_{k,\ell} : S_{\omega a} \subset B\}.$$ Then clearly
\begin{align*}
\#(A_{k,\ell}'') &\lesssim \#(A_{k,\ell}) |B|  \asymp \#(A_{k,\ell}') \frac{|B|}{|B_0|}
\end{align*}
and thus
\begin{align*}
\frac1\Omega \sum_{N_0 \leq k,\ell\leq N_1} \sum_{(\omega,a) \in A_{k,\ell}''} f(\rho_{k,\ell}) \lambda_{S_{\omega a}}(B)
&\lesssim \frac1\Omega \sum_{N_0 \leq k,\ell\leq N_1} \#(A_{k,\ell}') \frac{|B|}{|B_0|} f(\rho_{k,\ell})
= \frac{|B|}{|B_0|}.
\end{align*}
Now for all $(\omega,a)$ such that $S_{\omega a}\cap B \neq \emptyset$, we have either $S_{\omega a} \subset B$ or $B \subset S_{\omega a}$. If the latter case never holds, then we are done. Otherwise, we have
\[
\mu(B) = \frac1\Omega f(\rho_{k,\ell}) \lambda_{S_{\omega a}}(B),
\]
where $(\omega,a)\in A_{k,\ell}'$ is chosen so that $B \subset S_{\omega a}$. Since $\Omega \geq C$ and $\rho_{k,\ell} \asymp |S_{\omega a}|$, we have
\[
\mu(B) \lesssim \frac{f(\diam B)}{C}
\]
in this case.

Now let $B_1$ be an arbitrary ball, and let $\omega$ be the longest word such that $B \subset S_\omega$. If there exist distinct $n_1,n_2\in\N$ such that $u_\omega(1/n_i) \in B_1$, then let $b_1 = \lfloor 1/\max(B_1)\rfloor$ and $b_2 = \lceil 1/\min(B_1)\rceil$ and let $B$ be as above. Then $\diam(B_1) \asymp \diam(B)$ and $B_1 \subset B$, so it follows from the previous paragraph that
\[
\mu(B_1) \lesssim \frac{f(\diam B_1)}{C} \cdot
\]
On the other hand, if there do not exist such distinct $n_1,n_2$, then by the maximality of $\omega$ there exists one such $n\in\N$ such that $u_\omega(1/n) \in B_1$. Let $n_1 = n$ and $n_2 = n+1$, and let $b_i$ be maximal such that
\[
B_1 \subset u_{n_1}([0,1/b_1]) \cup u_{n_2}([0,1/b_2]).
\]
The argument of the previous paragraph shows that for all $i=1,2$,
\[
\mu(u_{\omega n_i}([0,1/b_i])) \lesssim \frac{f(\diam u_{\omega n_i}([0,1/b_i]))}{C}
\]
and on the other hand, the maximality of $b_i$ gives
\[
\diam(u_{\omega n_i}([0,1/b_i])) \asymp \diam(B_1).
\]
It follows that
\[
\mu(B_1) \lesssim \frac{f(\diam B_1)}{C} \cdot
\]
Hence the proof of Theorem \ref{thm2} is complete.
\end{proof}

\section{Proofs of the corollaries}

\subsection{Proof of Corollary \ref{HKWWthm}}
 Let $f$ be an ESL dimension function. We show that 
 \begin{equation*}
 \sum_{\substack{k\geq 1\\ Q=2^k }}  \  \sum_{ \substack{j\geq 1, A=2^j \\ A<\Psi(Q)} } Q^2 A f \left( \frac1{Q^2 A \Psi(Q A)}\right)\asymp \sum\limits_{q}qf\left(\frac{1}{{q^2\Psi({q})}} \right) 
 \asymp\sum\limits_{q}q \log\left(\Psi(q)\right)f\left(\frac{1}{{q^2\Psi({q})}} \right).
 \end{equation*}Recall the covering argument for $\mathcal A_1(\Psi)$ 
 \begin{align*}
|J_n(a_1,\dots,a_n)|&\le \frac{1}{\Psi(a_nq_{n-1})a_nq_{n-1}^2} \ll \frac{1}{\Psi(q_n)a_nq_{n-1}^2} \asymp \frac{1}{\Psi(q_n)q_{n-1}q_n}.\end{align*}
Next following the same method as in \cite[pp. 512-13]{HKWW}, we arrive at the
\begin{align*}\label{nick}
\mathcal{H}^f\big(\mathcal A_1(\Psi)\big)\le&2\liminf_{N\to \infty}\sum_{q\ge 2^{(N-1)/2}}\sum_{\frac{q}{\Psi(q)}<p<q}f\left(\frac{1}{pq\Psi(q)}\right).\end{align*}
Following \cite{HKWW}, the restriction on $p$ follows from the fact that, since $a_n<\Psi(q_{n-1})$, $$q_n\leq (a_n+1)q_{n-1}\ll\Psi(q_{n-1})q_{n-1}< \Psi(q_{n})q_{n-1}\Longrightarrow q_{n-1}>\frac{q_n}{\Psi(q_n)}.$$
It can be seen that if
 \begin{equation*}\label{eqn1}\sum_{q=1}^\infty\sum_{\frac{q}{\Psi(q)}< p\le q}f\left(\frac{1}{pq\Psi(q)}\right)<\infty,\end{equation*}
 then it readily follows from Proposition \ref{bclem} that $\mathcal H^f\big(\mathcal A_1(\Psi)\big)=0.$ Next we compare all these series by using the condition \eqref{ffm10},
 \begin{align*}
 \sum_{q=1}^\infty q\log(\Psi(q))f\left(\frac{1}{q^2\Psi(q)}\right)&= 
 \sum_{q=1}^\infty\sum_{\frac{q}{\Psi(q)}< p\le q}\frac{q}{p}f\left(\frac{1}{q^2\Psi(q)}\right)\\ 
 & \overset{\eqref{ffm10}}{\geq}    \sum_{q=1}^\infty\sum_{\frac{q}{\Psi(q)}< p\le q}f\left(\frac{1}{pq\Psi(q)}\right)\\ 
 &\leq \sum_{q=1}^\infty\sum_{1< p\le q}f\left(\frac{1}{pq\Psi(q)}\right)\\ 
 &\overset{\eqref{equiv}}{\asymp}\sum_{q=1}^\infty qf\left(\frac{1}{q^2\Psi(q)}\right) \\ 
 &\leq \sum_{q=1}^\infty q\log(\Psi(q))f\left(\frac{1}{q^2\Psi(q)}\right).
 \end{align*}

\subsection{Proof of Corollary \ref{cor1}} 
Recall that 
\begin{equation*}\label{eqnf2}
 \sum_{\substack{k\geq 1\\ Q=2^k }}  \  \sum_{ \substack{j\geq 1, A=2^j \\ A<\Psi(Q)} } Q^2 A f \left( \frac1{Q^2 A \Psi(Q A)}\right)\asymp \sum\limits_{ j<\log_2\Psi(2^k)} 2^{2k+j}f\left(\frac{2^{-(2k+j)}}{\Psi(2^{k+j})}\right).
\end{equation*}
The proof of the corollary follows if we show that, for the dimension function $f$ satisfying \eqref{fnesl}, the following two series are equivalent
\begin{equation}\label{claim1}\sum_{\substack{k\geq 1\\ Q=2^k }}  \  \sum_{ \substack{j\geq 1, A=2^j \\ A<\Psi(Q)} } Q^2 A f \left( \frac1{Q^2 A \Psi(Q A)}\right)\asymp \sum\limits_{q}q \log\left(\Psi(q)\right)f\left(\frac{1}{{q^2\Psi({q})}} \right).
\end{equation}

Since $x^{-1}f(x)$ is quasi-monotonic, i.e. it satisfies \eqref{ff2},  we have that
$$Af\left(\frac{x}{A}\right)\leq CA^2f\left(\frac{x}{A^2}\right).$$
Using this inequality we have
\begin{align*}
\sum_{\substack{k\geq 1\\ Q=2^k }}  \  \sum_{ \substack{j\geq 1, A=2^j \\ A<\Psi(Q)} } Q^2 A f \left( \frac1{Q^2 A \Psi(Q A)}\right)
&\ll
\sum_{\substack{k\geq 1\\ Q=2^k }}  \  \sum_{ \substack{j\geq 1, A=2^j \\ A<\Psi(Q)} } Q^2 A^2 f \left( \frac1{Q^2 A^2 \Psi(Q A)}\right)   \\&\asymp  \sum_{ \substack{j\geq 1, A=2^j \\ A<\Psi(R)} }               R^2 f \left(\frac1{R^2 \Psi(R)}\right)\notag \hfill{ \ \text{by setting} \ R=QA}
\\ &\asymp  \sum\limits_{q}q \log\left(\Psi(q)\right)f\left(\frac{1}{{q^2\Psi({q})}} \right).
\end{align*}
For the reverse inequality, let $f$ satisfy the condition \eqref{fnesl}.  For convenience, let  $f(x)=x(\log(1/x))^{a_1}$ with $a_1\geq 0$. Clearly the series are equivalence for $a_1=0$, therefore we assume that $a_1>0$.
 For this, we show that $f(x)\gg Af(x/A)$, where $x= \frac1{Q^2 A^2 \Psi(Q A)}$. 

Clearly $Ax\leq 1.$
\begin{align*}
&\Longrightarrow \log (A/x)\leq \log(1/x)^2\\
&\Longrightarrow \log(A/x)^{a_1}\leq 2^{a_1}\log(1/x)^{a_1}.
\end{align*}
Thus, 
$$x(\log(1/x))^{a_1}\gg A\frac{x}{A}(\log(A/x))^{a_1}.$$
Hence,
$$
\sum_{\substack{k\geq 1\\ Q=2^k }}  \  \sum_{ \substack{j\geq 1, A=2^j \\ A<\Psi(Q)} } Q^2 A f \left( \frac1{Q^2 A \Psi(Q A)}\right)\gg
\sum_{\substack{k\geq 1\\ Q=2^k }}  \  \sum_{ \substack{j\geq 1, A=2^j \\ A<\Psi(Q)} } Q^2 A^2 f \left( \frac1{Q^2 A^2 \Psi(Q A)}\right).$$


\subsection{Proof of  Corollary \ref{cor2}} Just like the proof of the Corollary \ref{cor1}, we show the equivalence of the series \eqref{claim1} for the approximating function $\Psi$ assuming a growth condition \eqref{condpsi}, i.e. 
 for all $ x>0$ and $Q>1$, $\Psi(Q^x)\asymp\Psi(Q)$. 
 In this case, we have
\begin{align*}\sum_{\substack{k\geq 1\\ Q=2^k }}  \  \sum_{ \substack{j\geq 1, A=2^j \\ A<\Psi(Q)} } Q^2 A f \left( \frac1{Q^2 A \Psi(Q A)}\right) &\asymp  \sum_{\substack{k\geq 1\\ Q=2^k }}  \  \sum_{ \substack{j\geq 1, A=2^j \\ A<\Psi(Q)} } Q^2 A f \left( \frac1{Q^2 A \Psi(Q A^{1/2})}\right)\\ &\asymp
 \sum_{\substack{k\geq 1\\ R=2^k }}  \  \sum_{ \substack{j\geq 1, A=2^j \\ A<\Psi(R)} }               R^2 f \left(\frac1{R^2 \Psi(R)}\right)\notag \hfill{ \ \text{by setting} \ R=QA^{1/2}}\\ &\asymp \sum\limits_{q}q \log\left(\Psi(q)\right)f\left(\frac{1}{{q^2\Psi({q})}} \right). \notag\end{align*}

\def\cprime{$'$} \def\cprime{$'$} \def\cprime{$'$} \def\cprime{$'$}
  \def\cprime{$'$} \def\cprime{$'$} \def\cprime{$'$} \def\cprime{$'$}
  \def\cprime{$'$} \def\cprime{$'$} \def\cprime{$'$} \def\cprime{$'$}
  \def\cprime{$'$} \def\cprime{$'$} \def\cprime{$'$} \def\cprime{$'$}
  \def\cprime{$'$} \def\cprime{$'$} \def\cprime{$'$} \def\cprime{$'$}
  \def\cprime{$'$}

\end{document}